\documentclass[english,a4paper,12pt]{amsart}
\usepackage[a4paper,lmargin=2cm,rmargin=2cm,tmargin=4cm,bmargin=4cm]{geometry}
\usepackage[centertags]{amsmath}
\usepackage{amsfonts}
\usepackage{amssymb}
\usepackage{amsthm}
\usepackage[colorlinks]{hyperref}
\usepackage{tikz-cd}

\newcommand{\Natural}{\mathbb N}

\newcommand{\abs}[1]{\left\vert#1\right\vert}
\newcommand{\set}[1]{\left\{#1\right\}}

\newcommand{\norm}[1]{\left\Vert#1\right\Vert}
\newcommand{\duality}[1]{\left\langle#1\right\rangle}

\newcommand{\Free}{{\mathcal F}}
\newcommand{\Lip}{{\mathrm{Lip}}}
\newcommand{\F}[1]{\mathcal{F}(#1)}
\newcommand{\conv}{\mathop\mathrm{conv}}
\newcommand{\ext}{\mathop\mathrm{ext}}

\theoremstyle{plain}
\newtheorem{thm}{Theorem}

\newtheorem{lem}[thm]{Lemma}
\newtheorem{prop}[thm]{Proposition}

\theoremstyle{definition}

\newtheorem{rem}[thm]{Remark}

\title{On exposed points of Lipschitz free spaces}

\author[C. Petitjean]{Colin Petitjean}\thanks{This work was supported by the French ``Investissements d'Avenir'' program, project ISITE-BFC (contract ANR-15-IDEX-03).}
\address[C. Petitjean]{Universit\'e Bourgogne Franche-Comt\'e, Laboratoire de Math\'ematiques UMR 6623, 16 route de Gray,
25030 Besan\c con Cedex, France}\email{colin.petitjean@univ-fcomte.fr}

\author[A. Proch\'azka]{Anton\'in Proch\'azka}
\address[A. Proch\'azka]{Universit\'e Bourgogne Franche-Comt\'e, Laboratoire de Math\'ematiques UMR 6623, 16 route de Gray,
25030 Besan\c con Cedex, France}
\email{antonin.prochazka@univ-fcomte.fr}



\date{October, 2018}
\begin{document}
\begin{abstract}
 In this note we prove that  a molecule $d(x,y)^{-1}(\delta(x)-\delta(y))$ is an exposed point of the unit ball of a Lispchitz free space $\F M$ if and only if the metric segment $[x,y]=\{z \in M \; : \; d(x,y)=d(z,x)+d(z,y) \}$ is reduced to $\{x,y\}$. This is based on a recent result due to Aliaga and Perneck\'a which states that the class of Lipschitz free
spaces over closed subsets of M is closed under arbitrary intersections when M has finite diameter.
\end{abstract}
\maketitle

\section{Introduction}

For a metric space $(M,d)$ with a distinguished point $0 \in M$, we let $\Lip_0(M)$ be the real Banach space of Lipschitz maps from $M$ to $\mathbb R$ which vanish at $0$. 
We recall that the norm of $f \in \Lip_0(M)$, denoted $\| f \|_L$, is the best Lipschitz constant of $f$, i.e.
$$
\norm{f}_L=\sup_{x\neq y \in M}\frac{\abs{f(x)-f(y)}}{d(x,y)}.
$$ 
Next, for $x \in M$, we let $\delta(x) \in \Lip_0(M)^*$ be Dirac measure, i.e. $\duality{\delta(x),f}=f(x)$. We then define the Lipschitz free space over $M$ to be the following closed subspace of $\Lip_0(M)^*$:
$$\Free(M) := \overline{\mathrm{span}}\{\delta(x) \; : \; x \in M\}. $$
It follows from the fundamental linearisation property of Lipschitz free spaces that $\Free(M)$ is a canonical predual of $\Lip_0(M)$ (see \cite{GK} for more details).
\medskip

In this note we are interested in extreme points and exposed points of the unit ball of Lipschitz free spaces. If $B_X$ denotes the unit ball of a Banach space $X$, we recall that $x \in B_X$ is an extreme point of $B_X$ whenever $x \not \in \conv(B_X \setminus \{x\})$. Next, $x$ is an exposed point of $B_X$ if there exists a linear functional $f \in X^*$ such that $f(x) > f(z)$ for every $z \in B_X \setminus \{ x \}$. In what follows, $\ext(B_X)$ denotes the set of extreme points of $B_X$ while $\exp(B_X)$ denotes the set of exposed points of $B_X$.  Is is readily seen that $\exp(B_X)\subset \ext(B_X)$.
\smallskip

The extremal structure of Lipschitz free spaces has already been investigated in a number of articles \cite{AG,AP, Duality, Daugavet, Weaver}.
In any such study a special attention is dedicated to the elements of $\Free(M)$ of the form $m_{xy}=\frac{\delta(x)-\delta(y)}{d(x,y)}$ which we call \emph{molecules} (and which are called elementary molecules in~\cite{AP}).
It is simply a matter of writing down the corresponding convex combination to see that $m_{xy} \in \ext(B_{\Free(M)})$ implies that $[x,y]=\set{x,y}$. However, it is only recently that Aliaga and Perneck\'a~\cite{AP} managed to prove that, for a complete $M$, the reverse implication is also valid. 
Here, using one of the ingredients of their proof, we show the following stronger result.

\begin{thm}\label{t:ExposedMain}
Let $M$ be a complete metric space and $p\neq q \in M$ satisfy $[p,q]=\set{p,q}$. 
Then $m_{pq}$ is an exposed point of $B_{\Free(M)}$.
It is exposed by the 
magic function 
\[f_{pq}(t):= \frac{d(x,y)}{2}\left(\frac{d(t,q)-d(t,p)}{d(t,q)+d(t,p)}-\frac{d(0,q)-d(0,p)}{d(0,q)+d(0,p)}\right).\]
\end{thm}

\section{Proof of the main result}

The authors of~\cite{AP} had the following important insight which is likely to have many more applications in analysis of Lipschitz free spaces.
\begin{prop}[Aliaga and Perneck\'a~\cite{AP}]\label{p:Intersection}
Let $M$ be a bounded complete metric space. Let $\set{M_\alpha\subset M:\alpha \in A}$ be a collection of closed subsets of $M$ containing $0$.
Then $$\bigcap_{\alpha \in A} \Free(M_\alpha)=\Free\Big(\bigcap_{\alpha \in A} M_\alpha\Big).$$
\end{prop}

For the proof of Theorem~\ref{t:ExposedMain} we will need further some notation and few lemmas.
Given a metric space $M$ we will set $\widetilde{M}:=M \times M \setminus \set{(x,x):x\in M}$ and $V=\set{m_{xy}:(x,y)\in \widetilde{M}}$ the set of molecules in $\Free(M)$.
The following folklore fact is also stated in disguise as Lemma~2.1 in~\cite{AP}.
The proof here is different from the one in~\cite{AP}.  
\begin{lem}\label{l:folklor}
Let $M$ be a metric space. 
Let us define
$Q:\ell_1(\widetilde{M})\to \Free(M)$ by $e_{(x,y)} \mapsto m_{xy}$ and linearly on $span\set{e_{(x,y)}}$.
Then $Q$ extends to an onto norm-one mapping.
\end{lem}
\begin{proof}
The fact that $\norm{Q}=1$ is clear so we can extend $Q$ to the whole space with the same norm. Let us call the extension $Q$ again. 
We will prove that $B_{\Free(M)}^O\subset Q(B_{\ell_1}^O)$, where $B_X^O$ denotes the open unit ball of a Banach space $X$.
For this it is enough to use Lemma 2.23 in~\cite{FHHMPZ}, i.e. we need to check that $B_{\mathcal F(M)}^O\subset \overline{Q(B_{\ell_1}^O)}$.
But we have $B_{\Free(M)}^O\subset \overline{\conv}(V) \subset \overline{Q(B_{\ell_1})}=\overline{Q(B_{\ell_1}^O)}$.
\end{proof}

The next lemma is standard.
\begin{lem}\label{l:convexity}
Let $a\in S_{\ell_1}$ and $b\in B_{\ell_\infty}$. Assume that $1-\alpha\varepsilon\leq \duality{a,b}$ for some $0<\alpha,\varepsilon<1$.
Denote $B=\set{n\in \Natural: \abs{b_n}\leq (1-\alpha)}$. 
Then $\sum_{n\in B}\abs{a_n}\leq \varepsilon$.
\end{lem}
\begin{proof}
We denote $G:=\Natural \setminus B$.
We have 
\[
\begin{aligned}
1-\varepsilon\alpha &\leq \sum_{n=1}^\infty a_nb_n\leq\sum_{n\in G} \abs{a_nb_n} +\sum_{n\in B} \abs{a_nb_n}\\
&\leq \sum_{n \in G} \abs{a_n}+(1-\alpha)\sum_{n\in B} \abs{a_n}\\
&\leq \sum_{n \in \Natural} \abs{a_n} -\alpha \sum_{n \in B} \abs{a_n}=1-\alpha \sum_{n\in B} \abs{a_n}.
\end{aligned}
\]
It follows that $\sum_{n\in B}\abs{a_n} \leq \varepsilon$.
\end{proof}

For a metric space $M$, points $p,q \in M$ and $\varepsilon>0$ we will denote
\[
[p,q]_\varepsilon:=\set{x\in M: d(p,x)+d(x,q)\leq \frac{1}{1-\varepsilon}d(p,q)}.
\]
The properties of the magic function collected in the following lemma have been proved already in~\cite{ikw2}.
\begin{lem}\label{lemma:IKWfunction} Let $(p,q)\in \widetilde{M}$. We have
\begin{enumerate}
\item $f_{pq}$ is Lipschitz and $\norm{f_{pq}}_{L}\leq 1$. 
\item Let $u\neq v \in M$ and $\varepsilon>0$ be such that $\frac{f_{pq}(u)-f_{pq}(v)}{d(u,v)}>1-\varepsilon$. Then both $u,v \in [p,q]_\varepsilon$. 
\item If $(u,v)\in \widetilde{M}$ and $\frac{f_{pq}(u)-f_{pq}(v)}{d(u,v)}=1$, then both $u,v\in [x,y]$.
\end{enumerate} 
\end{lem}

Let us remark at this point that if $[p,q]=\set{p,q}$, then $f_{pq}$ exposes $m_{pq}$ among molecules (immediate from Lemma~\ref{lemma:IKWfunction}~(3)) and also among those $\mu \in B_{\Free(M)}$ which have finite support (or more generally such that $\norm{\mu}=\norm{a}_1$ in the representation coming from Lemma~\ref{l:folklor}). 
The next lemma prepares the ground for the remaining cases.

\begin{lem}\label{l:QuotientTrick}
Let $M$ be a metric space with the base point $0=q$ and let $p\neq q \in M$ be such that $[p,q]=\set{p,q}$. 
Assume that $\mu \in B_{\Free(M)}$ satisfies $\duality{\mu,f_{pq}}=1$.
Then for every $\varepsilon, \alpha\in (0,\frac12)$ we have $\mu \in \Free([p,q]_\alpha)+2\varepsilon B_{\Free(M)}$.
\end{lem}
\begin{proof}
Let us observe right away that by the hypothesis $\norm{\mu}=1$.
Let $\varepsilon,\alpha \in (0,\frac12)$ be fixed.
By Lemma~\ref{l:folklor} there exist $a=(a_i) \in \ell_1$ and $(p_i), (q_i) \subset M$ such that $\mu=\sum_{i=1}^\infty a_i m_{p_iq_i}$ and $\norm{a}_1\leq \norm{\mu}+\frac{\varepsilon\alpha}{1-\varepsilon\alpha}$.
We have
\[
1-\varepsilon\alpha \leq \frac1{\norm{a}_1}=\duality{\frac{\mu}{\norm{a}_1},f_{pq}}= \sum_{i=1}^\infty \frac{a_i}{\norm{a}_1}\duality{m_{p_iq_i},f_{pq}}.
\]
Now if we denote 
$B=\set{i\in \Natural: \abs{\duality{m_{p_iq_i},f_{pq}}}\leq (1-\alpha)}$, then Lemma~\ref{l:convexity} yields that $\sum_{i\in B} \abs{\frac{a_i}{\norm{a}_1}}\leq \varepsilon$ and so
$\sum_{i\in B} \abs{a_i}\leq 2\varepsilon$.
It follows from Lemma~\ref{lemma:IKWfunction}~(2) that for every $i \in \Natural \setminus B$ we have $p_i,q_i \in [p,q]_\alpha$. 
The conclusion is now immediate.
\end{proof}

\begin{proof}[Proof of Theorem~\ref{t:ExposedMain}]
We can assume without loss of generality that $0=q$. Indeed, a change of the base point in $M$ induces a linear isometry between the corresponding Lipschitz free spaces which preserves the molecules.
Lemma~\ref{l:QuotientTrick} shows that if $\mu \in B_{\Free(M)}$ satisfies $\duality{\mu,f_{pq}}=1$ then $\mu \in \bigcap_{\alpha>0} \Free([p,q]_\alpha)$. Since $[p,q]_{1}$ is bounded, 
Proposition~\ref{p:Intersection} yields that $\mu \in \Free([p,q])=\Free(\set{p,q})$. This is a 1-dimensional vector space so $\mu=\pm m_{pq}$ but only the choice of the plus sign is reasonable.
\end{proof}

\begin{rem}
Apart from the obvious fact that Theorem~\ref{t:ExposedMain} strengthens and generalizes some of the results in~\cite{Duality} let us also point out that one of the proofs of the main result in~\cite{dkp} (i.e. the characterization of $M$ such that $\Free(M)=\ell_1(\Gamma)$) becomes now much simpler.
\end{rem}

\textbf{Acknowledgment. } The authors are grateful to Ram\'on Aliaga and Eva Perneck\'a for sending them their preprint.


\begin{thebibliography}{bb}
\bibitem{AG} R. J. Aliaga and A. J. Guirao, \textit{On the preserved extremal structure of Lipschitz-free spaces}. Studia Math. 245 (2019), 1--14.

\bibitem{AP} R. Aliaga and E. Perneck\'a, \textit{Supports and extreme points in Lipschitz-free spaces}.  	arXiv:1810.11278 [math.FA].

\bibitem{dkp} A.~Dalet, P.~L.~Kaufmann and A.~Proch\'azka, \emph{Characterization of metric spaces whose free space is isometric to $\ell_1$}. Bull. Belg. Math. Soc. Simon Stevin 23 (2016), no. 3, 391-400.

\bibitem{FHHMPZ} M.~Fabian, P.~Habala, P.~H\'ajek, V.~Montesinos Santaluc\'ia, J.~Pelant, V. Zizler, \textit{Functional analysis and infinite-dimensional geometry}. CMS Books in Mathematics/Ouvrages de Mathématiques de la SMC, 8. Springer-Verlag, New York, 2001.

 \bibitem{Duality} L. Garc\'ia-Lirola, C. Petitjean, A. Proch\'azka, A. Rueda Zoca, \textit{Extremal structure and duality of Lipschitz free spaces}.  Mediterr. J. Math. 15 (2018), no. 2, Art. 69, 23 pp. \href{http://arxiv.org/abs/1707.09307}{arxiv 1707.09307}
 
 \bibitem{Daugavet} L. Garc\'ia-Lirola, A. Proch\'azka, A. Rueda Zoca,  \textit{A characterisation of the Daugavet property in spaces of Lipschitz functions}. J. Math. Anal. Appl. 464 (2018), no. 1, 473?492. \href{https://arxiv.org/abs/1705.05145}{arxiv 1705.05145}
 
\bibitem{GK} G. Godefroy and N. J. Kalton, Lipschitz-free Banach spaces, Studia Math. 159 (2003), 121--141.
 
\bibitem{ikw2} Y.~Ivakhno, V.~Kadets and D.~Werner, \textit{Corrigendum to: The Daugavet property for spaces of Lipschitz functions}, Math. Scand. \textbf{104} (2009), 319-319.


\bibitem{Weaver} N.~Weaver, Lipschitz Algebras, World Scientic Publishing Co., River Edge, NJ, 1999.
\end{thebibliography}
\end{document}